\documentclass{amsart}

\title{Self-similar measures with unusual local dimension properties}
\author{Kevin G. Hare}
\address{Department of Pure Mathematics, University of Waterloo, Waterloo, Ontario, Canada N2L 3G1}
\email{kghare@uwaterloo.ca}
\thanks{Research of K.G. Hare was supported by NSERC Grant 2019-03930}

\usepackage{amssymb,latexsym}
\usepackage{amsmath}
\usepackage{amsthm}
\usepackage{comment}
\usepackage{xcolor}
\usepackage{hyperref}
\usepackage{graphicx}

\newtheorem{thm}{Theorem}[section]
\newtheorem{lemma}[thm]{Lemma}
\newtheorem{notation}[thm]{Notation}
\newtheorem{prop}[thm]{Proposition}
\newtheorem{example}[thm]{Example}

\newtheorem{remark}[thm]{Remark}
\newtheorem{corollary}[thm]{Corollary}
\newtheorem{defn}[thm]{Definition}
\newcommand{\dimloc}{\dim_{\mathrm{loc}}}
\newcommand{\supp}{\mathrm{supp}}

\begin{document}
\begin{abstract}
Let $\mu$ be a self-similar measure satisfying the finite type condition.
It is known that the set of attainable local dimensions for such a measure
    is a union of disjoint intervals, where some intervals may be degenerate
    points.
Despite this, it has not been shown if this full complexity of 
    attainable local dimensions is achievable.
In this paper we give two different constructions.
The first is a measure $\mu$  where the set of all attainable local dimensions 
    is the union of an interval union and an arbitrary number of disjoint
    points.
The second is a measure $\mu$ where the set of all attainable local dimensions 
    is the union of an arbitrary number of disjoint intervals.
As an application to these construction, we study the multi-fractal 
    spectrum $f_\mu(\alpha)$ and the $L^q$-spectrum $\tau_\mu(q)$ of these measures.
We given an example of a $\mu$ where $f_\mu(\alpha)$ is not concave, and 
    where $\tau_\mu(q)$ has two points of non-differentiability. 
\end{abstract}
\maketitle

\section{Introduction}

Let $\mu$ be a finite Borel measure.  
The {\em local dimensions of $\mu$ at a point $x$} is a way to quantify the concentration of the 
    measure at $x$.
If the limit exists, it is defined as 
\begin{equation*}
\dimloc\mu (x)=\lim_{r\rightarrow 0^{+}}\frac{\log \mu
([x-r,x+r])}{\log r}.
\end{equation*}
A similar concept of upper local dimension and lower local dimension can 
    also be defined.
See Definition \ref{defn:dimloc}.

In the next section we will define what it means for a measure $\mu$ to be of finite 
    type.
For measures of finite type, satisfying some technical assumptions, it is relatively straightforward 
    to determine the (upper/lower) local dimension at a point $x \in \supp \mu$.
Further, it is often possible to say meaningful things about 
    the set of all attainable local dimensions.
It is often, but not always the case that this set is a interval.
The set of all attainable local dimensions for such a measure is always a union of disjoint intervals (where the intervals may be degenerate points) 
    \cite{HR}.

Examples of self-similar measures are known where the set of all attainable local dimensions is
\begin{itemize}
\item An interval: This is true for any self-similar measure satisfying the 
    open set condition.  See for instance \cite{Fa}. 
\item An interval and a disjoint point: This is true for the $m$-fold 
    convolution ($m \geq 3$) of the Cantor measure. See \cite{HL}.
\item Two disjoint intervals. See \cite[Section 7]{BHM} or \cite[Section 6]{Testud} for example.
\end{itemize}

Our first goal is to give two different constructions of self-similar 
    measure so that the set of attainable local dimensions have a higher
    degree of complexity than previously known examples.
The first construction is a measure $\mu$  where the set of all attainable local dimensions is the union 
    of an interval and an arbitrary number of disjoint points.
The second is a measure $\mu$ where the set of all attainable local dimensions is the union
    of an arbitrary number of disjoint intervals.

Our second goal is to apply these construction to the study of the multi-fractal spectrum and $L^q$-spectrum of $\mu$.
We define the {\em multi-fractal spectrum of $\mu$} as 
    \[ f_\mu(\alpha) = \dim_H \{ x \in \supp \mu: \dimloc \mu (x) = \alpha\}.\]
We further define the {\em $L^q$-spectrum of $\mu$} as 
    \[ \tau_{\mu}(q) = \liminf_{r \to 0} \frac {\log \sup \sum_i \mu (B(x_i, r))^q}{\log r}. \]
It is worth remarking that some $L^q$-spectrum is sometimes defined as the negative of this function.
The study of these two functions, and their relationship has a long history.
See for instance \cite{CM92, Fa, Lau99, NX, Pa97, Riedi95, Rutar} and 
    references therein.
Using these constructions, we can construct $\mu$ such that $f_\mu(\alpha)$ 
    is non-concave.

The study of the differentiability of $\tau_\mu(q)$ is also of interest,
    and in general poorly understood.
When $\mu$ satisfies the open set condition, then $\tau_\mu(q)$ is known to be 
    differentiable for all $q$.
Further it is known that $f_\mu(\alpha) = \tau_\mu^*(\alpha) = \inf_q (\alpha q - \tau(q))$ when
    $\mu$ satisfies the open set condition.
This relationship between $f_\mu(\alpha)$ and $\tau_\mu(q)$, when it exists, is known 
    as the {\em multi-fractal formalism}.
Special cases of this problem were studied in \cite{F1, LN1, LN2} for $\mu$ that were of finite type,
    but not satisfying the open set condition.
In these cases, $\tau_\mu(q)$ was proven to be differential, despite not satisfying the open set condition.
It was shown in \cite{F3} that for $q > 0$ that $\tau_\mu(q)$ is differentiable.
Examples of self-similar measures of finite type $\mu$ were given in \cite{F1} and \cite{LW} where their existed a point 
    where $\tau_\mu(q)$ was not differentiable.
In \cite{Testud}, a construction was given where $\tau_\mu(q)$ had an arbitrary number of 
    non-differentiable points.
The method used was to subdivide $\supp(\mu)$ into two sets, $A$ and $B$, 
    and considered the related functions $\tau_A(q)$ and $\tau_B(q)$ 
    restricted to these set.
It was then shown that $\mu$ could be constructed such that  $\tau_\mu(q) = \min(\tau_A(q),\tau_B(q))$,
     and that $\tau_A(q) = \tau_B(q)$ had an arbitrary number
     of solutions. 
Each of these solutions resulted in a point of non-differentiability.
See also \cite{DN} and references for further history on this problem.

In Section \ref{sec:basic} we will give some basic definitions and terminology 
   used in this paper.
In Section \ref{sec:intuition} we will give the intuition behind the 
    constructions given in Section \ref{sec:construction}.
In Section \ref{sec:Lq} we discuss the application of these constructions
    to the study to the multi-fractal spectrum and the $L^q$-spectrum.
We give an example of a function such that $f_\mu(\alpha)$ is non-concave, and 
    further when $\tau_\mu(q)$ has two points of non-differentiability.

We make some final observations and raise some questions in Section \ref{sec:final}.
Also in this section, we sketch how one might be able to find an alternate construct a $\mu$ such that $\tau_\mu(q)$ has an arbitrary
    number of non-differentiable points.

\section{Basic Definitions and Terminology}
\label{sec:basic}

In this section we will give the basic definitions of (upper/lower) local dimensions,
    and what it means for a measure to be of finite type.
We then introduce the terminology for characteristic vectors and transition 
    matrices that are necessary for the study of local dimensions of measures of finite type.
We will restrict the study to measures of finite type satisfying what is called the 
    standard technical conditions, as they are sufficient for our constructions.
Measures that are of finite type, but do not satisfy the standard technical
    conditions are well studied.
See for instance \cite{HHN18, HHR21, HHS18}. 

\begin{defn}
\label{defn:dimloc}
Given a finite Borel measure $\mu $, by the \textbf{upper local dimension} of 
$\mu $ at $x\in \supp \mu $, we mean the number 
\begin{equation*}
\overline{\dimloc}\mu (x)=\limsup_{r\rightarrow 0^{+}}\frac{\log \mu
([x-r,x+r])}{\log r}.
\end{equation*}
Replacing the $\limsup $ by $\liminf $ gives the \textbf{lower local
dimension}, denoted $\underline{\dimloc}\mu (x)$. If the limit exists, we
call the number the \textbf{local dimension} of $\mu $ at $x$ and denote
this by $\dimloc\mu (x)$.
\end{defn}

By an {\em iterated function system (IFS)}, we mean a finite set of contractions $\{S_i\}$.
Each IFS generates a unique non-empty compact set $K$ such that 
    $K = \cup S_i(K)$.
This is known as the {\em attractor} of the iterated function system.

In this paper, we will focus our attention on equicontractive IFS on $\mathbb{R}$.
That is, the set of contractions are of the form
\begin{equation}
S_{j}(x)=r x+d_{j}:\mathbb{R\rightarrow R}\text{ for }j=0,1,\dots,k,
\label{IFS}
\end{equation}
where $k\geq 1$ and $0< r <1$. 
We will assume that the $S_j$ are ordered so that 
    $d_0 < d_1 < \dots < d_k$.
By rescaling the $d_{j},$ if necessary, we can assume the convex hull of $K$
is $[0,1]$.
In this paper we will further assume that $K = [0,1]$.  

We will associate to $\{S_i\}$ the probabilities $p_i > 0$ with $\sum p_i = 1$.
There is a unique self-similar measure $\mu$ such that 
\begin{equation}
\label{eq:mu}
\mu =\sum_{j=0}^{k}p_{j}\mu \circ S_{j}^{-1}.
\end{equation}
This measure is non-atomic, and is supported on the attractor of the associate IFS, 
    in this case $[0,1]$.

We begin by introducing the notion of finite type and the related concepts
    and terminology that will be used.
A more complete description can be found in \cite{F3, F1, F2, HHM16, HHS18}.

Let $\{S_i\}_{i=0}^k$ be a family of equicontractive maps with ratio of contraction $r$.
We define $\Lambda_n$ as the set of words of length $n$ on the 
    alphabet $\{0,1,\dots, k\}$.
We further define  $\Sigma = \cup_{n=0}^\infty \Lambda_n$ as the set of all finite words 
    over this alphabet.
For $\sigma = (\sigma_1 \sigma_2 \dots \sigma_n) \in \Lambda_n$ we define 
    \[ S_{\sigma} = S_{\sigma_1} \circ \dots \circ S_{\sigma_n}. \]
For $\sigma$ of length $n$, we have that $S_\sigma([0,1]) \subset [0,1]$ is an subinterval of length $r^n$.

Let $\sigma \in \Lambda_n$ be a proper prefix of the word $\tau \in \Lambda_m$ for $n < m$.
As $K = [0,1]$ we see that $S_\tau([0,1])\subsetneq S_\sigma([0,1])$.
We associate to an infinite word $(\sigma_1, \sigma_2, \dots, ) \in \{0,1,\dots, k\}^{\mathbb{N}}$ 
    the unique point $x$ given by 
    \[ \cap_{n = 1}^\infty S_{\sigma_1 \dots \sigma_n}([0,1]).\]
This will be known as an {\em address} of $x$.
It is possible for $x$ to have more than one address.

The finite type condition for iterated function systems was first introduced by Ngai and Wang in \cite{NW}.
The definition we will use is slightly less general, as we are considering 
only equicontractive measures with attractor $K = [0,1]$.
There are many equivalent formulations of this condition.

\begin{defn}
Assume $\{S_{j}\}$ is an IFS as in equation \eqref{IFS}. The words $\sigma ,\tau
\in \Lambda _{n}$ are said to be neighbours if $S_{\sigma }(0,1)\cap S_{\tau
}(0,1)\neq \emptyset $. Denote by $\mathcal{N}(\sigma )$ the set of all
neighbours of $\sigma $. We say that $\sigma \in \Lambda _{n}$ and $\tau \in
\Lambda _{m}$ have the same neighbourhood type if there is a map $f(x)= r^{n-m}x+c$ such that 
\begin{equation*}
\{f\circ S_{\eta }:\eta \in \mathcal{N}(\sigma )\}=\{S_{\nu }:\nu \in 
\mathcal{N}(\tau )\} \ \ \mathrm{and}\ \ f \circ S_\sigma = S_\tau.
\end{equation*}
The IFS is said to be of \textbf{finite type} if there are only finitely
many neighbourhood types.
\end{defn}

We say that the self-similar measure $\mu$, as in \eqref{eq:mu} is of 
    finite type if the set of contractions $\{S_i\}$ are of finite 
    type.
Given $\sigma = (\sigma_1, \sigma_2, \dots, \sigma_n) \in \Sigma$, 
   we define \[ p_{\sigma }=\prod_{i=1}^{n}p_{\sigma _{i}}. \]

As in \cite{HHM16}, we will say the self-similar measure $\mu$ satisfies the {\em standard technical assumptions}
    if it is of finite type, $p_0 = p_k = \min p_i$ and $K = \supp\mu = [0,1]$.
In this paper, we will only consider self-similar measures satisfying the standard technical assumptions,
    as they are sufficient for our constructions.

\begin{defn}
For each positive integer $n$, let $h_{1},\dots ,h_{s_{n}}$ be the
collection of elements of the set $\{S_{\sigma }(0),$ $S_{\sigma }(1):\sigma
\in \Lambda _{n}\}$, listed in increasing order. Put 
\begin{equation*}
\mathcal{F}_{n}=\{[h_{j},h_{j+1}]:1\leq j\leq s_{n}-1 \}.
\end{equation*}
Elements of $\mathcal{F}_{n}$ are called \textbf{net intervals of generation 
}$n$. The interval $[0,1]$ is understood to be the only net interval of
generation $0$.
\end{defn}

If $\Delta =[a,b] \in \mathcal{F}_n$ we see that $a$ and $b$ are of the form 
    $S_\sigma(0)$ or $S_\sigma(1)$ for some $\sigma \in \Lambda_n$.
As $S_\sigma(0) = S_\sigma \circ S_0 (0)$ and $S_\sigma(1) = S_\sigma \circ S_k(1)$ 
    we see that $\mathcal{F}_{n}$ is a refinement of $\mathcal{F}_{n-1}$.
That is, there exists a unique $\widehat \Delta \in \mathcal{F}_{n-1}$ such 
    that $\Delta \subset \widehat \Delta$.
We call $\Delta$ the {\em child} of $\widehat \Delta$, and $\widehat \Delta$ the 
    {\em parent} of $\Delta$.

For $\Delta = [a,b] \in \mathcal{F}_n$, we define the {\em normalized length} 
    of $\Delta$ as 
\begin{equation*}
\ell _{n}(\Delta )=r^{-n}(b-a)\text{.}
\end{equation*}
We define the {\em neighbour set} of $\Delta$ as the ordered tuple 
\begin{equation*}
V_{n}(\Delta )=(a_{1},a_{2},\dots ,a_{m(\Delta)}),
\end{equation*}
where for each $i$ there is some $\sigma \in \Lambda _{n}$ such that $r^{-n}(a-S_{\sigma}(0))=a_{i}$.
Abusing terminology slightly, we will refer to one of these $a_{i}$ as a {\em neighbour} of $\Delta$.
 
Suppose  $\Delta \in \mathcal{F}_n$ has parent $\widehat \Delta$.
It is possible for $\widehat \Delta$ to have multiple children with the 
    same normalized length and the same neighbourhood set as $\Delta$.
We will indicate by $t_n(\Delta)$ which of the identical children $\Delta$ is.

\begin{defn}
The \textbf{characteristic vector} of $\Delta \in \mathcal{F}_{n}$
is defined to be the tuple 
\begin{equation*}
\mathcal{C}_{n}(\Delta )=(\ell _{n}(\Delta ),V_{n}(\Delta ), t_n(\Delta)).
\end{equation*}
The characteristic vector of $[0,1]$ is defined as $(1,(0),1)$.
\end{defn}

We often suppress the $t_n(\Delta)$ to give the {\em reduced characteristic vector}.
It is worth noting that the set of children of $\Delta$ depends only upo  
    $\ell_n(\Delta)$ and $V_n(\Delta)$, and does not depend upon $t_n(\Delta)$.

By the {\em symbolic representation} of a net interval $\Delta \in 
\mathcal{F}_{n}$ we mean the $n+1$ tuple $(\mathcal{C}_{0}(\Delta _{0}),\dots,
\mathcal{C}_{n}(\Delta _{n}))$ where $\Delta _{0}=[0,1]$, $\Delta
_{n}=\Delta $, and for each $j=1,\dots,n$, $\Delta _{j-1}$ is the parent of $
\Delta _{j}$. Similarly, for each $x\in \lbrack 0,1]$, the symbolic
representation of $x$ will be the sequence of characteristic vectors
\begin{equation*}
\lbrack x]=(\mathcal{C}_{0}(\Delta _{0}),\mathcal{C}_{1}(\Delta _{n}),\dots)
\end{equation*}
where $x\in \Delta _{n}\in \mathcal{F}_{n}$ for each $n$ and $\Delta _{j-1}$
is the parent of $\Delta _{j}$. 

We will denote the set of characteristic vectors by 
\begin{equation*}
\Omega =\{\mathcal{C}_{n}(\Delta ):n\in \mathbb{N}\text{, }\Delta \in 
\mathcal{F}_{n}\}\text{.}
\end{equation*}

As $\mu$ is of finite type, we have $\Omega$ is a finite set.

As in \cite{HHM16}, we
define {\em primitive transition matrices}, $T(\mathcal{C}_{n-1}( 
\widehat{\Delta }),$ $\mathcal{C}_{n}(\Delta )),$ for a net interval $\Delta
=[a,b]$ of level $n$ and parent $\widehat{\Delta }=[c,d]$ as follows:

\begin{notation}
\label{not:T}
Suppose $V_{n}(\Delta )=\left(a_{1},\dots,a_{m(\Delta)}\right)$ and $V_{n-1}(\widehat{\Delta }
)=\left(c_{1},\dots,c_{m(\widehat{\Delta})}\right)$. For $j=1,\dots,m(\widehat{\Delta})$ and $k=1,\dots,m(\Delta),$ we set 
\begin{equation*}
T_{jk}:=\left( T(\mathcal{C}_{n-1}(\widehat{\Delta }),\mathcal{C}_{n}(\Delta
))\right) _{jk}=p_{\ell }\text{ }
\end{equation*}
if $\ell \in \mathcal{A}$ and there exists $\sigma \in \mathcal{A}^{n-1}$
with $S_{\sigma }(0)=c-{r} ^{n-1}c_{j}$ and $S_{\sigma \ell
}(0)=a-{r} ^{n}a_{k}$. This is equivalent to saying 
\begin{equation*}
T_{jk}=p_{\ell }\text{ if }c-{r} ^{n-1}c_{j}+{r}
^{n-1}d_{\ell }=a-{r} ^{n}a_{k}.
\end{equation*}
We set $\left( T(\mathcal{C}_{n-1}(\widehat{\Delta }),\mathcal{C}_{n}(\Delta
))\right) _{jk}=0$ otherwise.
\end{notation}

\begin{remark}
It is worth remarking that in \cite{HHM16} that these matrices were normalized
    by dividing by $p_0$.
To be consistent with later works,  have chosen to not do this, and as a result 
    a number of the stated results are modified to adjust for this change.
\end{remark}

It is also worth noting here that $\widehat \Delta$ may have multiple children
    with the same reduced characteristic vector $\Delta$.  
The associated transition matrix will depend on which child $\Delta$ 
    we are transitioning to, and hence depends on $t_n(\Delta)$.

It can be shown in this case that
\begin{itemize}
\item There are only a finite number of such matrices.
\item All non-zero entries are bounded below by $p_0$.
\item There is a non-zero entry in every row and every column of every
    transition matrix
\end{itemize}

Consider a sequence of characteristic vectors $\theta = (\gamma_1, \gamma_2, \dots, \gamma_n)$
     where each $\gamma_{i+1}$ is a child of $\gamma_i$.
We define the {\em transition matrix of $\theta$} as 
\[ T(\theta) = T(\gamma_1, \gamma_2) T(\gamma_2, \gamma_2) \dots T(\gamma_{n-1}, \gamma_n).\]

With this notation, 
\begin{corollary}[Corollary 3.10 of \cite{HHM16}]
\label{corlocdim} Suppose $\mu $ satisfies the standard technical
assumptions. If $x\in \supp \mu $, then 
\begin{equation}
\overline{\dimloc}\mu (x)=\limsup_{n
\rightarrow \infty }\frac{\log \left\Vert T([x|n])\right\Vert }{n\log
{r} }  \label{F2}
\end{equation}
and similarly for the (lower) local dimension.
\end{corollary}

\begin{remark}
This result is equivalent to \cite{HHM16}, and is adjusted to account for the 
    lack of normalization in Notation \ref{not:T}.
\end{remark}

A {\em periodic point} $x$ is a point with symbolic representation
\begin{equation*}
\lbrack x]=(\gamma _{1},\dots,\gamma _{J},\theta ^{-},\theta ^{-},\dots),
\end{equation*}
where $\theta =(\theta _{1},\dots,\theta _{s},\theta _{1})$ is a cycle
(meaning, the first and last letters are the same) and $\theta ^{-}$ has the
last letter of $\theta $ deleted. We call $\theta $ a {\em period} of $x$.

Note that $T(\theta )$ is a square matrix as $\theta $ is a cycle. We denote by $
sp(T(\theta ))$ its spectral radius, the largest eigenvalue in absolute 
    value of $T(\theta)$.
We have
\begin{prop}[Proposition 4.14 of \cite{HHM16}]
\label{periodic}If $x$ is a periodic point with period $\theta $ of period
length $\beta $, then the local dimension of $\mu $ at $x$ exists and is
given by 
\begin{equation*}
\dimloc    \mu (x)=\frac{\log sp(T(\theta ))}{\beta \log r}.
\end{equation*}
\end{prop}

\begin{remark}
Again, this has been adjusted from \cite{HHM16} to account for the 
    different normalization in Notation \ref{not:T}.

If $x$ is a boundary point, then it will have two different periodic 
    representations.
These formulae are valid for both representations, and are equal.
\end{remark}

We next define a loop class, maximal loop class and essential class.

\begin{defn}
\begin{itemize}
\item A non-empty subset $\Omega ^{\prime }\subseteq \Omega $ is called a
\textbf{\ loop class} if whenever $\alpha ,\beta \in \Omega ^{\prime }$,
then there are reduced characteristic vectors $\gamma _{j}$, $j=1,\dots,n$, such
that $\alpha =\gamma _{1}$, $\beta =\gamma _{n}$ and $(\gamma
_{1},\dots,\gamma _{n}) $, with each $\gamma_{j+1}$ a child of $\gamma_j$, 
    and all $\gamma _{j}\in \Omega ^{\prime }$.

\item A loop class $\Omega ^{\prime }\subseteq \Omega $ is called an\textbf{\
essential class} if, in addition, whenever $\alpha \in \Omega ^{\prime }$
and $\beta \in \Omega $ is any child of $\alpha $, then $\beta \in \Omega
^{\prime }$.

\item We call a loop class \textbf{maximal }if it is not properly contained
in any other loop class.
\end{itemize}
\end{defn}

\begin{remark}
It is worth remarking that in \cite{HHM16} the loop class, maximal loop class and
    essential class were defined as a subset of characteristic vectors.
We have instead chosen to define them as a subset of reduced characteristic vectors.
These two methods are equivalent.
\end{remark}

It is shown in \cite[Lemma 6.4]{F2} that there is always  a unique 
    essential class.

\begin{defn}
If $[x]=(\gamma _{0},\gamma _{1},\gamma _{2},\dots)$ with $\gamma _{j}\in
\Omega _{0}$ for all large $j$, we will say that $x$ is an \textbf{
essential point} (or is \textbf{in the essential class}) and call $x$ a
\textbf{non-essential point} otherwise. The phrase, $x$ is \textbf{in the
loop class }$\Omega ^{\prime },$ will have a similar meaning. An admissible
path will be said to be in a given loop class if all its members are in that
class.
\end{defn}

\section{Intuition to the constructions}
\label{sec:intuition}

In Section \ref{sec:construction} we will give the two construction of $\mu$ such that they
    have unusual local dimension properties.
In this section, we will give the intuition behind these constructions, as well as a simple example.

Consider the self-similar measure given by the IFS 
    $S_j(x) = \frac{x}{R} + \frac{j}{R^2}$ for 
    $j = 0, 1, \dots, R (R-1)$.
Assume that $p_0 = p_{R(R-1)} = \min p_j$.
This measure satisfies the standard technical assumptions.

Some observations that are worth making about this measure are
\begin{itemize}
\item We have that 
    \[ \mathcal{F}_n = \left\{\left[\frac{i}{R^{n+1}},\frac{i+1}{R^{n+1}}\right]
    \right\}_{i=0}^{R^n (R-1) -1} \]
    for $n \geq 1$.
    As such, all reduced characteristic vectors, aside from the initial reduced characteristic
    vector of $[0,1]$ in $\mathcal{F}_0$ are of normalized length $1/R$.
\item For $i \geq R$ and $i \leq R^n (R-1) - R$ we see that the 
    reduced characteristic vector associated to 
    $\left[\frac{i}{R^{n+1}},\frac{i+1}{R^{n+1}}\right]$ is 
    \[ v_e = (1/R, (0,1/R, 2/R, \dots, (R-1)/R)).\]
\item We see that the children of $v_e$ are $R$ copies of $v_e$.
    Hence the essential class is the singleton $\{v_e\}$.
\item There is a loop class at $0$ and a loop class at $1$.  All other 
    points $x \in (0,1)$ are in the essential class.
\end{itemize}

The construction that we will give uses this measure as a starting point,
    and remove a number of the contractions (adjusting the probabilities 
    as necessary).
We will remove contractions in such a way that the remaining contractions
    can be partitions into disjoint sets 
    $\mathcal{A}, \mathcal{B}_1, \dots, \mathcal{B}_d$, where these 
    sets have some very specific properties.

Requirements:
\begin{enumerate}
\item
\label{req:1}
We require that \[ \cup_{S \in \mathcal{A} \cup \mathcal{B}_1 \cup \dots \cup \mathcal{B}_d} S([0,1]) = [0,1]. \]
That is, we require the measure based on this subset of contractions 
    will be of full support, and hence still satisfy the standard technical
    assumptions..
\item 
\label{req:2}
For each $i$, we have that $\# \mathcal{B}_i = 1 \text{ or } 2$.
This is done for simplicity of the construction, and it should be possible to 
choose $\mathcal{B}_i$ to be larger subsets.
\item 
\label{req:4}
For each $i$, define $K_i$ such that $K_i = \cup_{S \in \mathcal{B}_i} S(K_i)$.
This is the attractor associated to $\mathcal{B}_i$.
If $\#\mathcal{B}_i = 1$ then $K_i$ will be a singleton, whereas it 
   $\#\mathcal{B}_i = 2$ then $K_i$ will be a Cantor set 
   with Hausdorff dimension $\log(2)/\log(R)$.

For each $K_i$ and $S \in \mathcal{A}$ we require
\[ S((0,1))\cap K_i = \emptyset.\]
We further require for each $K_i$ and each $S \in \mathcal{B}_j$ with $j \neq i$ that 
\[ S((0,1))\cap K_i = \emptyset.\]
This ensures that if $x \in K_i$ and $x$ is not a boundary point then 
    $x$ has an unique address, and this address only uses 
    contractions from $\mathcal{B}_i$.
If $x \in K_i$ and $x$ is a boundary point, then $x$ may also have addresses
    with tails $S_0^\infty$ or $S_{R (R-1)}^\infty$.  
As the measure satisfies the standard technical assumptions, these addresses
    do not contribute to the local dimension of $x$.
In particular this implies that if $x \in K_i$ then the local dimension of 
    $\mu(x)$ depends only on the probabilities associated with 
    $S \in \mathcal{B}_i$.
\item 
\label{req:5}
If $x \not \in K_i$ for all $i$ then $x$ is in the essential class.
Typically this is done by choosing $\mathcal{A}$ sufficiently dense so that 
    the essential class is the singleton $\{v_e\} = \{(1/R, (0, 1/R, 2/R, \dots, (R-1)/R))\}$
    and $\mathcal{A}$ as large as possible while 
    still ensuring that requirements \eqref{req:1} - \eqref{req:4} hold.
\end{enumerate}

\begin{example}
Let $R = 4$.
We let $S_j(x) = \frac{x}{16} + \frac{j}{16}$.
Define 
\begin{itemize}
\item $\mathcal{B}_1 = \{S_0\}$ with $p_0 = \frac{1}{164}$.  We see in this case that $K_1 = \{0\}$.
\item $\mathcal{B}_2 = \{S_6\}$ with $p_6 = \frac{2}{164}$.  We see in this case that $K_1 = \{1/2\}$.
\item $\mathcal{B}_3 = \{S_{12}\}$ with $p_{12} = \frac{1}{164}$.  We see in this case that $K_1 = \{1\}$.
\item $\mathcal{A} = \{S_j\}$ with $j \in \{1,2,3,4,8,9,10,11\}$ and $p^* = \frac{20}{164}$.
\end{itemize}
One can easily verify properties \eqref{req:1} - \eqref{req:4} hold.
To see that property \eqref{req:5} holds in this particular case,
    we compute the set of all reduced characteristic vectors using the techniques
    of \cite{HHM16}.
They are 
\[ 
\begin{array}{lll}
v_1 = (1, (0)) & 
v_2 = (1/4, (0)) & 
v_3 = (1/4, (0, 1/4))  \\
v_4 = (1/4, (0, 1/4, 1/2)) &
v_5 = (1/4, (0, 1/4, 1/2, 3/4)) & 
v_6 = (1/4, (1/4, 1/2, 3/4)) \\
v_7 = (1/4, (0, 1/2, 3/4)) & 
v_8 = (1/4, (1/4, 3/4)) &
v_9 = (1/4, (0, 1/2)) \\
v_{10} = (1/4, (0, 1/4, 3/4)) & 
v_{11} = (1/4, (1/2, 3/4)) & 
v_{12} = (1/4, (3/4))
\end{array}
\]
Below we list the children, in order, for each reduced characteristic vector
\[
\begin{array}{lll}
\multicolumn{3}{l}{ v_1 \to v_2, v_3, v_4, v_5, v_5, v_6, v_7, v_8, v_9, v_{10}, v_4, v_5, v_5, v_6, v_{11}, v_{12}} \\
v_2 \to v_2, v_3, v_4, v_5, & 
v_3 \to v_5, v_5, v_5, v_5,  &
v_4 \to v_5, v_5, v_5, v_5,  \\
v_5 \to v_5, v_5, v_5, v_5, &
v_6 \to v_5, v_5, v_5, v_5,&
v_7 \to v_5, v_5, v_5, v_5, \\
v_8 \to v_5, v_6, v_7, v_8, &
v_9 \to v_9, v_{10}, v_4, v_5, & 
v_{10} \to v_5, v_5, v_5, v_5,\\
v_{11} \to v_5, v_5, v_5, v_5,  &
v_{12} \to v_5, v_6, v_{11}, v_{12}, 
\end{array}
\]

It is easy to see that the essential class is the single reduced 
    characteristic vector $\{v_5\}$.

We see that there are 4 maximal loop classes outside of the essential class.

The first maximal loop class is $\{v_2\}$.
Points in this loop class have symbolic representation $(\mathcal{C}_0(\Delta_0), 
    \mathcal{C}_1(\Delta_1), \dots) = (v_1, v_2, v_2, v_2, \dots)$. 
This is associated to the point $x = 0$ or equivalently $x \in K_1$.
This point has address $0,0,0,\dots$.
The transition matrix for this loop class is $T(v_2,v_2) = [1/164]$, and hence
    the local dimension at this point is 
    $\frac{\log(1/164)}{\log(1/4)} \approx 3.678776003$.

The second and third maximal loop classes are $\{v_8\}$ and $\{v_9\}$.
Points in this loop class have symbolic
    representation  $(v_1, v_8, v_8, v_8, \dots)$ and $(v_1, v_9, v_9, v_9, \dots)$.
These are both associated to the point $ x= 1/2$, or equivalently $ x \in K_2$.
It is worth noting that because $x$ is a boundary point, there are two
     different representations.
This point has three addresses,
    namely address $8,0,0,0,\dots,$, the address $6,6,6,\dots$ and the address $4,12,12,12,\dots$.
We see that the middle representation contributes the majority of 
    the weight.
The transition matrix for these loop classes are 
\begin{align*}
T(v_8,v_8) & =  \frac{1}{164} \left[ \begin {array}{cc} 2&20\\ 0&1\end {array} \right]  \\
T(v_9,v_9) & =  \frac{1}{164} \left[ \begin {array}{cc} 1&0\\ 20&2\end {array} \right] 
\end{align*}
We note that both of these periods are of period 1, and both have spectral
    radius $\frac{2}{164}$.
As such, the local dimension at this point is
    $\frac{\log(2/164)}{\log(1/4)} \approx 3.178776003$.

The last maximal loop class is $\{v_{12}\}$.
Points in this loop class have symbolic representation $v_1 , v_{12}, v_{12}, v_{12}, \dots$. 
This is associated to the point $x = 1$ or equivalently $x \in K_3$.
This point has address $12,12,12,\dots$.
The transition matrix for this loop class is $T(v_{12},v_{12}) = [1/164]$, and hence
    the local dimension at this point is 
    $\frac{\log(1/164)}{\log(1/4)} \approx 3.678776003$.

All other points are in the essential class.  
Using the techniques of \cite{HHM16} we can show that 
\begin{align*}
[.989157974, 1.017811955]
& \subset \{\dimloc \mu(x): x \text{ in the essential class}\} \\
& \subset [.983436074, 1.017811955]. 
\end{align*}

From this we get that the set of attainable local dimensions is the 
    union of an interval (associated to those $x$ in the essential class)
    and two disjoint points (associated to $K_1$ and $K_3$ for one of these
    two points, and to $K_2$ for the second of these two points).
\end{example}

In this example we exploit the fact that it was a specific case, and did 
    direct computations to ensure that the measure satisfied requirements 
    \eqref{req:1} to \eqref{req:5}.
We also were able to verify that the choice of probabilities 
    ensured that the local dimension of $x \in K_i$ were 
    sufficiently separated by each other, and from the local dimensions
    found in the essential class.
In the next section we will discuss the general construction, where we are 
    not able to do this.
Hence we will need to 
\begin{itemize}
\item Ensure that either $x \in K_i$ for some $i$ or $x$ is in the 
      essential class.
\item Ensure that the local dimension of $x \in K_i$ are bounded away from 
      the local dimension of $x \in K_j$ for $i \neq j$ and bounded away 
      from the local dimension of $x$ in the essential class.
\item Show that if $\#\mathcal{B}_i = 2$ then the set of attainable local 
      dimensions of $K_i$ is an interval, depend only upon the probabilities
      associated to the maps in $\mathcal{B}_i$.
\end{itemize}

\section{Construction}
\label{sec:construction}
\subsection{Isolated points with isolated local dimensions}
\label{ssec:multipoints}

Our goal is to construct a measure $\mu$ such that the set of attainable
     local dimensions is the disjoint union of an interval and an 
     arbitrary number of disjoint points.
Let $R$ be even.
We will do this by constructing a measure such that the essential
    class is $[0,1] \setminus \{0,2/R, 4/R, \dots, (R-2)/R,1\}$.
The points $2i/R$ for $i = 0, 2, \dots, R/2$ are all associated to their own 
    loop classes.
The probabilities can be chosen such that the local dimension at these
    loop points are all distinct (except the two end points $0$ and $1$ which are equal) 
    and bounded away from the local dimensions coming from the essential class.

For $i = 0, 1, \dots, R/2$ we will define $t_i = 2 i (R - 1)$.
We will define $\mathcal{B}_i = \{S_{t_i}\}$ where 
    $S_{t_i}(x) = \frac{x}{R} + \frac{t_i}{R^2}$.
We will associate to $S_{t_i}$ the probability $p_{t_i}$.
We noticed that the attractor for $\mathcal{B}_i$ is 
    $K_i = \left\{\frac{2i}{R}\right\}$.
We will define $\mathcal{A} = \{S_i\}_{i \in I}$ where
$S_i(x) = x/ R + i/R^2$ and 
    $I = \{i: \forall j\ S_i((0,1)) \cap K_j = \emptyset \}$.
We associate to each $S \in \mathcal{A}$ the same probability, $p^*$,
    chosen so that the sum over all probabilities is $1$.
As usual we require $p_0 = p_{t_{R/2}} = \min p_i$.

We see that $\cup_{S \in \mathcal{A} \cup \mathcal{B}_0 \cup \mathcal{B}_{R/2}} S([0,1]) = [0,1]$, from 
    which is follows that requirement \eqref{req:1} holds.
Requirement \eqref{req:2} holds trivially.

We see that $\frac{2j}{R} \not\in S_{t_i}(0,1) = \left(\frac{2i (R-1)}{R^2},\frac{2i(R-1)+R}{R^2}\right)$
    for $i \neq j$.
By construction, we see that for all $S \in \mathcal{A}$ that $\frac{2j}{R} \not\in S(0,1)$.
From this it follows that requirement \eqref{req:4} holds.

It remains to show requirement \eqref{req:5}.
\begin{thm}
\label{thm:multipoints EC}
For $\mu$ as described above, either $x \in K_i$ for some $i$, or 
    $x$ is in the essential class.
\end{thm}

\begin{proof}
We first sketch the major steps of the proof, after which we will
    complete the proof.

Steps:
\begin{enumerate}
\item We first show that the essential class is composed of a 
    single reduced characteristic vector, namely $\{v_e\} = \{(1/R, (0, 1/R, \dots, (R-1)/R))\}$.
\label{step:1}
\item We next show that for any reduced characteristic vector $v = (1/R, A)$ where 
    $\{0, 1/R\} \subseteq A$ or $\{(R-2)/R, (R-1)/R\} \subseteq A$ we have
    $v$ has $R$ children, all of them $v_e$.
\label{step:2}
\item We lastly show that either $x \in K_i$ for some $i$, or $x$ is an 
    essential point.
\label{step:4}
\end{enumerate}

{\bf Step \ref{step:1}}:
By construction the minimal normalized length of a reduced characteristic
    vector is $1/R$.
By considering the children of $\Delta_0 = [0,1]$, we see that the 
    child located at $[1/R, 1/R + 1/R^2]$ is of the form 
    $v_e := (1/R, (0, 1/R, 2/R, \dots, (R-1)/R))$.
Further we observe that the children of $v_e$ are $R$ copies of $v_e$.
As the essential class is unique, 
    we have that the essential class is the singleton reduced
    characteristic vector $\{v_e\}$.

{\bf Step \ref{step:2}}:
This step is a routine calculation.

{\bf Step \ref{step:4}}:
Consider the children of $[0,1]$.
We will classify these reduced characteristic vector as follows.
We let $v_0'$ be the reduced characteristic vector associated to the 
    net interval adjacent to $0$ on the right, where $\{0\} = K_0$.
For $i = 1, 2, \dots, R/2-1$ we let $v_i$ and $v_i'$ be the reduced 
    characteristic vectors associated to the net intervals adjacent to 
    $2i/R$ on the left and right respectively, where 
    $K_i = \{2i/R\}$.
We let $v_{R/2}$ be the reduced characteristic vector associated to the 
    net interval adjacent to $1$ on the left, where $\{1\} = K_{R/2}$.
We note that all other net intervals are of the form $(1/R, A)$ where
    either $\{0, 1/R\} \subset A$ or $\{(R-2)/R, (R-1)/R\} \subset A$.

We see that $v_0' = (1/R, (0))$.
For $i = 1, 2, \dots, R/2-1$, we 
    see that $v_i = (1/R, ((2i-1)/R, (R-1)/R))$ and
    $v_i' = (1/R, (0, 2i/R))$.
We see that $v_{R/2} = (1/R, ((R-1)/R))$.

We can check that the children of $v_i$ are of the form 
    $(1/R, A)$ where $\{(R-2)/R, (R-1)/R)\} \subset A$ with the 
    exception of the far right child, which is a copy of $v_i$.
Similarly, the children of $v_i'$ are of the form 
    $(1/R, A)$ where $\{0, 1/R\} \subset A$ with the 
    exception of the far left child, which is a copy of $v_i'$.

This proves that either $x \in K_i$ for some $i$, or $x$ is an 
    essential point, as desired.
\end{proof}

\begin{lemma}
Let $t_i = 2 i (R - 1)$ as before.  
For $x \in K_i$ we have that $x$ has local dimension $\frac{\log p_{t_i}}{\log 1/R}$.
\end{lemma}

\begin{proof}
Using the notation of $v_i$ and $v_i'$ as above, we see that 
\begin{align*}
T(v_0', v_0') & = \left[ \begin {array}{c}  p_0 \end {array} \right]  \\
T(v_i, v_i) & =   \left[ \begin {array}{cc} p_{t_i} & p^* \\ 0 & p_0 \end {array} \right]  \\
T(v_i', v_i') & = \left[ \begin {array}{cc} p_0 & 0 \\ p^* & p_{t_i} \end {array} \right]  \\
T(v_{R/2}, v_{R/2}) & =   \left[ \begin {array}{c} p_{R/2} \end {array} \right] 
\end{align*}
By the standard technical assumptions, we have that $p_0 = p_{R/2} \leq 
    p_{t_i}$.
For $x \in K_i$ we have that
\[ \dimloc \mu(x) = \frac{\log sp(T(v_i, v_i)}{\log(1/R)}
                  = \frac{\log sp(T(v_i', v_i')}{\log(1/R)}
                  = \frac{\log p_{t_i}}{\log (1/R)} \]
as required.
\end{proof}

\begin{thm}
\label{thm:multipoints}
There exists a choice of probabilities $p_0, p_1, \dots, p_{R/2}$ and 
    $p^*$ such that the set of attainable local dimensions of $\mu$ is 
    a disjoint union of an interval, (coming from the essential class)
    and $R/2$ singletons.
Moreover the singletons are of the form $\frac{\log(p_{t_i})}{\log(1/R)}$.
    for $i = 0, 1, \dots, R/2$ with 
    $p_{t_0} = p_{t_{R/2}}$..
\end{thm}

\begin{proof}
As we can take $p_{t_{i}}$ arbitrarily small,, we may
     assume that $p^* \geq p_{t_{i}}$ for $i = 0, 1, \dots, R/2$.
Let $x$ be a point in the essential class.
Let $\widehat \Delta$ be a net interval 
    for $x$ with child $\Delta$ where 
    $\mathcal{C}_{n}(\Delta) = \mathcal{C}_{n-1}(\widehat \Delta) = v_e$.
Let $\Delta$ be the $r$th child of $\widehat \Delta$.
(That is, $t_n(\Delta) = r$).
Consider the transition matrix 
    $T := T(\mathcal{C}_{n-1}(\widehat\Delta), \mathcal{C}_n(\Delta))$.
We observe that $T_{1,k} = p^*$ for $k = 1, 2, \dots, R-r$ and 
    $T_{2, k} = p^*$ for $k =  R-r+1, R-r+2, \dots, R-1$.
This implies that 
    $sp(T(\mathcal{C}_{n-1}(\widehat\Delta),\mathcal{C}_n(\Delta))$ is 
    bounded below by $p^*$, for all transitions matrices from the 
    essential class to the essential class.
This in turn implies that the set of attainable local dimensions 
    for $x$ in the essential class is bounded above by 
    $\frac{\log(p^*)}{\log(1/R)}$.
As $p^*$ is bounded below by $1/\#(\mathcal{A} \cup 
    \mathcal{B}_0 \cup \dots \cup \mathcal{B}_{R/2})$ we have an upper bound
    for attainable local dimensions of $x$ in the essential class.
Let this upper bound by $B$.

We see that $p_{t_i}$ can be taken arbitrarily small, and hence, the 
    local dimensions for $x \in K_i$ can be taken to be arbitrarily 
    large.
We choose $p_{t_i}$ sufficient small so that $\dimloc \mu (x) >B$
    for all $x \in K_i$.
The result follows by taking $p_1, p_2, \dots, p_{R/2}$ distinct and 
    $p_0 = p_{R/2}$.
\end{proof}

\subsection{Cantor sets with an interval of local dimensions}
\label{ssec:multi-intervals}

Our goal is to construct a measure $\mu$ such that the set of attainable
     local dimensions is the disjoint union of an arbitrary
     number of intervals.
Let $R \equiv 2 \mod 6$.
We will do this by constructing a measure such that the non-essential
    points are $0$, $1$ and $(R-2)/6$ Cantor sets $K_i$ with endpoints
    $(6i-4)/R$ and $6i/R$.
Each of these $K_i$ will be the attractor of a set $\mathcal{B}_i$, 
    and will be associated to its own loop class,
The local dimensions within this loop class is an interval dependent only
    upon two probabilities associated with $S \in \mathcal{B}_i$.
The probabilities can be chosen such that the local dimension at these
    $K_i$ are all distinct intervals and bounded away from the local dimensions coming from the essential class.

We define $\mathcal{B}_0 = \{S_0\}$ where $S_0(x) = x /R$.
We associate to $S_0$ the probability $p_0$.
We see that the attractor of $\mathcal{B}_0$ is $K_0 = \{0\}$.

Similarly we define $\mathcal{B}_{(R+4)/6} = \{S_{R(R-1)}\}$ where $S_{R(R-1)}(x) = x /R + R(R-1)/R^2$.
We associate to $S_{R(R-1)}$ the probability $p_{R(R-1)}$.
We see that the attractor of $\mathcal{B}_{(R+4)/6}$ is $K_{(R+4)/6} = \{1\}$.

For $i = 1, \dots, (R-2)/6$ we will define 
    $t_i = (6i-4) (R - 1)$ and $s_i = 6i (R - 1)$.
We will define $\mathcal{B}_i = \{S_{t_i}, S_{s_i}\}$ where 
    $S_{t_i}(x) = \frac{x}{R} + \frac{t_i}{R^2}$, and $S_{s_i}$ is defined similarly.
We will associate to $S_{t_i}$ the probability $p_{t_i}$ 
                 and to $S_{t_i}$ the probability $p_{s_i}$.
We noticed that the fixed point of $S_{t_i}$ is $(6i-4)/R$ and 
    the fixed point of $S_{s_i}$ is $6i/R$.
We see that $K_i$ is a middle $(R-2)/R$ Cantor set with end points
    $(6i-4)/R$ and $6i/R$.
In particular, it is contained within 
    \[
    \left[\frac{R(6i-4)}{R^2}, \frac{R(6i-4)+4}{R^2}\right] \cup
     \left[\frac{R6i-4}{R^2}, \frac{R6i}{R^2}\right].\]

As before, we define we will define $\mathcal{A} = \{S_i\}_{i \in I}$ where
$S_i(x) = x/ R + i/R^2$ and 
    $I = \{i: \forall j\ S_i((0,1)) \cap K_j = \emptyset\}$.
We associate to each $S \in \mathcal{A}$ the same probability, $p^*$,
    chosen so that the sum over all probabilities is $1$.
As usual we require $p_0 = p_{R(R-1)} = \min p_i$.

One can check that requirement \eqref{req:1} holds.
Requirement \eqref{req:2} holds trivially.

We see that $K_j \subset \left[\frac{6j-4}{R}, \frac{6j}{R}\right]$, 
    $S_{t_i}(0,1) = \left(\frac{(6i-4) (R-1)}{R^2},\frac{(6i-4)(R-1)+R}{R^2}\right)$ and 
    $S_{s_i}(0,1) = \left(\frac{6i (R-1)}{R^2},\frac{6i(R-1)+R}{R^2}\right)$.
A quick computation ensures that 
    $K_j \cap S_{t_i}(0,1) = K_j \cap S_{s_i}(0,1) = \emptyset$ for $i \neq j$.
By construction, we see that for all $S \in \mathcal{A}$ that $K_i \cap S(0,1) = \emptyset$.
From this it follows that requirement \eqref{req:4} holds.

It remains to show requirement \eqref{req:5}.

\begin{thm}
\label{thm:multi-intervals EC}
For $\mu$ as described above, either $x \in K_i$ for some $i$, or 
    $x$ is in the essential class.
\end{thm}

\begin{proof}
As before, the essential class is composed of a 
    single reduced characteristic vector, namely $\{v_e\} = \{(1/R, (0, 1/R, \dots, (R-1)/R))\}$.
Further, any reduced characteristic vector of normalized length $1/R$ of the 
    form  $v = (1/R, A)$ where 
    $\{0, 1/R\} \subseteq A$ or $\{(R-2)/R, (R-1)/R\} \subseteq A$ 
    will have $R$ children, all of them $v_e$.

Consider the children of $[0,1]$.
We will classify these reduced characteristic vector as follows.
We let $v_0'$ be the reduced characteristic vector associated to the 
    net interval adjacent to $0$ on the right, where $\{0\} = K_0$.
We let $v_{(R+4)/6}$ be the reduced characteristic vector associated to the 
    net interval adjacent to $1$ on the left, where 
    $\{1\} = K_{(R+4)/6}$.

For $i = 1, 2, \dots, (R-2)/6$ 
    we let $v_i$ be the reduced 
    characteristic vectors associated to the net intervals 
    $\left[(6 i - 4)/R - 1/R^2, (6i-4)/R\right]$.
This is adjacent on the left to the convex hull of $K_i$.
This has the form $(1/R, ((6i-5)/R, (R-1)/R))$
We let $c_i$ be the reduced 
    characteristic vectors associated to the net intervals 
    $\left[(6 i - 4)/R, 6i/R\right]$.
This is the convex hull of $K_i$.
This has the form $(4/R, (6i-4)/R)$.
Lastly, we let $v_i'$ be the reduced 
    characteristic vectors associated to the net intervals 
    $\left[6i/R, 6i/R+1/R^2\right]$.
This is adjacent on the right to the convex hull of $K_i$.
This has the form $(1/R, (0, 6i/R))$.

We note that all other net intervals are of the form $(1/R, A)$ where
    either $\{0, 1/R\} \subset A$ or $\{(R-2)/R, (R-1)/R\} \subset A$.

We can check that the children of $v_i$ are of the form 
    $(1/R, A)$ where $\{(R-2)/R, (R-1)/R)\} \subset A$ with the 
    exception of the far right child, which is a copy of $v_i$.
Similarly, the children of $v_i'$ are of the form 
    $(1/R, A)$ where $\{0, 1/R\} \subset A$ with the 
    exception of the far left child, which is a copy of $v_i'$.

The two left most children of $c_i$ are of the form $c_i, v_i'$.
The two right most children of $c_i$ are also of them form $v_i, c_i$.
All other children of $c_i$ are of the form 
    $(1/R, A)$ where $\{(R-2)/R, (R-1)/R)\} \subset A$ or 
    $\{0, 1/R\} \subset A$.

This proves that either $x \in K_i$ for some $i$, or $x$ is an 
    essential point, as desired.
\end{proof}

\begin{thm}
Let $t_i = (6i-4) (R - 1)$ and $s_i = 6i (R - 1)$.
The set of attainable local dimension 
    of $x\in K_i$ is \[ \left[\frac{\log \max(p_{s_i},p_{t_i})}{\log 1/R},
       \frac{\log \min(p_{s_i},p_{t_i})}{\log 1/R}\right] .\]
\end{thm}

\begin{proof}
Let $c_i$ be as above.
We see that $\{c_i\}$ is a (maximal) loop class, and that the two 
    transition matrices from $c_i$ to itself are 
\[ T(c_i, c_i) = \left[ p_{t_i} \right] \text{  and  }
   T(c_i, c_i) = \left[ p_{s_i} \right]\]
There exist boundary points in $K_i$ which also have representations 
    with tails 
    $v_i', v_i', v_i', \dots$ and $v_i, v_i, v_i, \dots$.
From the remark following Proposition \ref{periodic} the local dimension of boundary points  
    based $T(v_i', v_i')$ and $T_(v_i, v_i)$ will the same as the periodic 
    representation coming form $T(c_i, c_i)$.
The result follows.
\end{proof}

\begin{thm}
There exists a choice of probabilities $p_0, p_1, \dots, p_{(R+4)/6}$ and 
    $p^*$ such that the set of attainable local dimensions of $\mu$ is 
    a disjoint union of an intervals.
\end{thm}

This follows using the same techniques as Theorem \ref{thm:multipoints}.

\section{Implications for the multi-fractal spectrum and $L^q$-spectrum}
\label{sec:Lq}

A self-similar measure is known to be {\em exact dimensional}.  
That is, there exists a unique $\alpha_0$ such that $\mu$-almost 
    all $x \in \supp \mu$ have local dimensions $\alpha_0$.
For $\alpha \neq \alpha_0$, the set of $x$ such that $\dimloc \mu(x) = \alpha$
    is necessarily of $\mu$-measure $0$.
This being said, this set can often be quite complicated, and have positive
    Hausdorff dimension.
To that end, we define the {\em multi-fractal spectrum of $\mu$} as
    \[ f_\mu (\alpha) = \dim_H\{ x \in \supp \mu : \dimloc \mu(x) = \alpha\}. \]
See for instance \cite{Fa}.

We further define the {\em $L^q$-spectrum of $\mu$} as 
    \[ \tau_{\mu}(q) = \liminf_{r \to 0} \frac {\log \sup \sum_i \mu (B(x_i, r))^q}{\log r} \]
where the supremum is taken over all disjoint families of balls of radius 
    $r$, centered at $x_i \in \supp \mu$.
If $f_\mu(\alpha)$ is concave, and 
    \[ f_\mu(\alpha) = \tau^*_\mu(\alpha) = \inf_{q \in \mathbb{R}} (\alpha q - \tau_\mu(q)) \] 
then we say that $\mu$ satisfies the {\em multi-fractal formalism}.
Here $\tau^*_\mu(\alpha)$ is the Legendre transform of $\tau_\mu(q)$.
This is known to hold for example when $\mu$ satisfies the open set condition.
In the case of our constructions, we often have that $f_{\mu}(\alpha)$ is not concave, 
    and that these measure do not satisfy the multi-fractal formalism.

Despite this, we can often say quite precise things about both $f_\mu(\alpha)$ and 
    $\tau_\mu(q)$, as we will show in the next example.

\begin{example}
\label{ex:Lq}
Consider the construction in Section \ref{ssec:multi-intervals} with $R = 14$.

Set $p_0 = p_{13\cdot 14} = \frac{1}{1150}$.
We see that $p_0$ is associated to $K_0 = \{0\}$, which has
    Hausdorff dimension $0$.
We see that $p_{182}$ is associated to $K_3 = \{1\}$, which has
    Hausdorff dimension $0$.
The local dimension at $x = 0$ and at $x = 1$ is 
    $\frac{\log(1/1150)}{\log(1/14)} \approx 2.67$.

We set $p_{s_1} = p_{t_1} = \frac{3}{1150}$,
This will be associate to the Cantor set $K_1$ with end points at
    $\frac{2}{14}$ and $\frac{6}{14}$.
This Cantor set has a ratio of contraction of $1/14$, and has 
    Hausdorff dimension $\frac{\log 2 }{\log 14}$.
As $p_{s_1} = p_{t_1}$ we see that the local dimensions for $K_1$
    are all equal, and are equal to $\frac{\log (3/1150)}{\log(1/14)}
    \approx 2.25$.

We set $p_{s_2} = \frac{5}{1150}$ and $p_{t_2} = \frac{7}{1150}$.
This will be associate to the Cantor set $K_2$ with end points 
    at $\frac{8}{14}$ and $\frac{12}{14}$.
As before, $K_2$ has
    Hausdorff dimension $\frac{\log 2 }{\log 14}$.
We see that the set of attainable local dimensions for $K_2$
    is an interval,
    \[ \left[\frac{\log(7/1150)}{\log(1/14)}, 
       \frac{\log(5/1150)}{\log(1/14)} \right] \approx
      [1.93, 2.06]. \]

For all maps $S_i \in \mathcal{A}$ we set $p^* = \frac{10}{1150}$.
One can check that the sum of the probabilities is equal to $1$.

Using the techniques of \cite{HHM16} we can show that 
\begin{align*}
[.9913,1.009]
& \subset \{\dimloc \mu(x): x \text{ in the essential class}\} \\
& \subset [.9666,1.038]
\end{align*}

Let us consider $f_\mu(\alpha)$ and $\tau_\alpha(q)$ for this measure.

We see that $f_\mu(\alpha) = 0$ for $\alpha \in (0,0.9666) \cup (1.038, 1.93) 
    \cup (2.06, 2.25) \cup (2.25, \infty)$.
We see that $f_\mu(\alpha) = 1$ for some $\alpha \in (0.9666,1.038)$.
This is coming from the fact that $\mu$ is exact dimensional, and 
    $K = [0,1]$ has Hausdorff dimension 1.

We see that $f_\mu(\alpha)$ is a concave function on $[1.93, 2.06]$, with a 
    maximum of $\log(2)/\log(14)$ coming from the dimension of $K_2$.
This can be computed explicitly by Chapter 11 of \cite{Fa}, and is
    given in Figure \ref{fig:K1}.

\begin{figure}
\includegraphics[scale=0.3, angle=270]{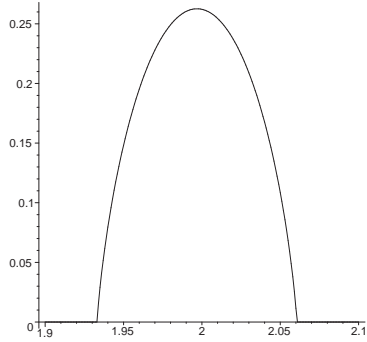}
\caption{$f_\mu(\alpha)$ associated to $K_2$}
\label{fig:K1}
\end{figure}

We see that we have a jump discontinuity at 
    $\alpha = \log(3/1150)/\log(1/14) \approx 2.25 $.
Namely 
\[ f_\mu\left(\frac{\log(3/1150)}{\log(1/14)}\right) = \frac{\log(2)}{\log(14)}\] and 
\[ \lim_{\alpha \to \frac{\log(3/1150)}{\log(1/14)}} f_{\mu}(\alpha) = 0.\]
This is coming from $K_1$.

As $K_0 = \{0\}$ and $K_3 = \{1\}$ both have Hausdorff dimension $0$, 
    we have that $f_\mu\left(\frac{\log(1/1150)}{\log(1/14)}\right) = 0$.
    
Using the techniques of \cite{HHS19}, we see that 
\[ 
\tau_\mu(q) = 
\min(\tau_E(\mu, q), 
\tau_{K_0}(\mu, q), 
\tau_{K_1}(\mu, q),  
\tau_{K_2}(\mu, q), 
\tau_{K_3}(\mu, q)).  \]
where $\tau_A(\mu, q)$ is the $L^q$-spectrum restricted to a set $A$ 
    where $A$ is the set of points coming from a (maximal) loop class.
Here $E$ is the essential class.
See Definition 4.4 of \cite{HHS19} for a precise definition.

With the exception of $\tau_E(\mu, q)$, all $\tau_{K_i}(q)$ can be explicitly
    computed using the techniques of Chapter 11 of \cite{Fa},
    we get that 
\begin{align*}
\tau_{K_0}(q) & = -\frac{\log\left((1/1150)^q\right)}{\log 14} \\
\tau_{K_1}(q) & = -\frac{\log\left(2\cdot (3/1150)^q\right)}{\log 14} \\
\tau_{K_2}(q) & = -\frac{\log\left((5/1150)^q+(7/1150)^q\right))}{\log 14} \\
\tau_{K_3}(q) & = -\frac{\log\left((1/1150)^q\right)}{\log 14} 
\end{align*}

In the case of $\tau_E(\mu,q)$ we have from Proposition 4.6 of \cite{HHS19}
that 
\begin{align*}
q d_{\min} & \geq \tau_E(\mu, q)  \geq q d_{\min} - 1  \text{  if } q \geq 0\\
q d_{\max} & \geq \tau_E(\mu, q)  \geq q d_{\max} - 1  \text{  if } q \leq 0
\end{align*}
where $d_{\min} = \inf \{ \dimloc\mu(x): x \in \text{essential class}\}$ and 
   $d_{\max}$ is defined similarly as the supremum.
In this case we do not have a precise value for $d_{\min}$ or $d_{\max}$.
We have that $d_{\min} \in [0.9666, 0.9913]$ and $d_{\max} \in [1.009, 1.038]$.
Hence we have that 
\begin{align*}
0.9913 q & \geq \tau_E(\mu, q)  \geq 0.9666 q - 1  \text{  if } q \geq 0\\
1.009 q & \geq \tau_E(\mu, q)  \geq 1.038 q - 1  \text{  if } q \leq 0
\end{align*}

We know that $\tau_{\mu}(0)$ is equal to negative of the box 
    dimension of the self similar set, in this case $-1$.

By Corollary 4.11 of \cite{HHS19} we have that 
    $\tau_E(q)/q \to d_{\min}$ as $q\to \infty$ and
    $\tau_E(q)/q \to d_{\max}$ as $q\to -\infty$.

As $\tau_E(q)$ is concave, this gives us much tighter bounds.
Namely
\begin{align*}
\min(1.038q-1,0.9913 q) & \geq \tau_E(\mu, q)  \geq 0.9666 q - 1  \text{  if } q \geq 0\\
\min(0.9666q-1,1.009 q) & \geq \tau_E(\mu, q)  \geq 1.038 q - 1  \text{  if } q \leq 0
\end{align*}

\begin{figure}
\includegraphics[scale=0.3, angle=270]{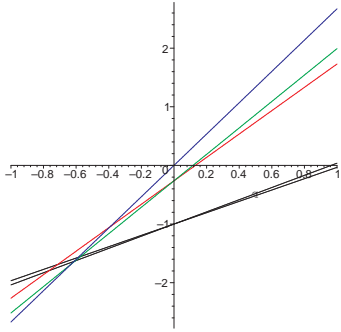}
\caption{$\tau_\mu(q)$ for $q \in [-1,1]$}
\label{fig:Lq}
\end{figure}

\begin{figure}
\includegraphics[scale=0.3, angle=270]{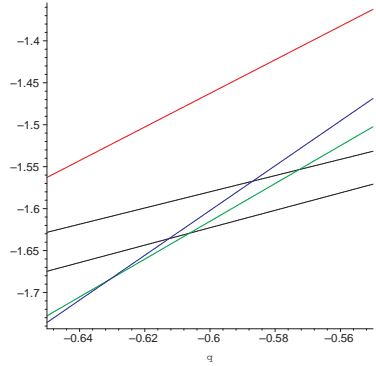}
\caption{$\tau_\mu(q)$ for $q \in [-0.65,-0.55]$}
\label{fig:LqSmall}
\end{figure}

In Figure \ref{fig:Lq}, we have that the blue curve is coming from $K_0$ and $K_3$, 
    the green curve is coming from $K_1$, and the red curve is coming from $K_2$.
We have that the curve $\tau_E(\mu, q)$ is bounded between the two black curves.

In Figure \ref{fig:LqSmall}, we provide the same map, except restricted to $q \in (-0.65,-0.55)$.
We notice that there exists two points $q_0$ and $q_1$ such that 
\begin{align*}
\tau_{\mu}(q) & = \tau_{K_0}(q) \text{ if } q \leq q_0 \\
\tau_{\mu}(q) & = \tau_{K_1}(q) \text{ if } q_0 \leq q \leq q_1 \\
\tau_{\mu}(q) & = \tau_{E}(q) \text{ if } q_1 \leq q 
\end{align*}
For no values of $q$ is $\tau_{\mu(q)} = \tau_{K_2}(q)$.

We see that $\tau_\mu(q)$ has distinctly 
    different characteristics on $(-\infty, q_0), (q_0, q_1)$ and $(q_1, \infty)$.
We observe that the slope of $\tau_\mu(q)$ on these three regions is 
    bounded by the maximal and minimal local dimensions associate with $K_0$, $K_1$ 
    and the essential class respectively.
As these local dimensions are bounded away from each other, we see that $q_0$ and $q_1$ are 
    both points of non-differentiability.
\end{example}

\begin{remark}
It is worth highlighting the difference between the multiple points of non-differentiability from
    this construction, and that coming from \cite{Testud}.
In \cite{Testud}, the interval $[0,1]$ was subdivided into two regions, $A$ and $B$.
Testud consider the case when the ratio of contraction was an odd or an even integer.
We mention the even case only.  The odd is similar.
In this case $B$ was a Cantor formed by the $\ell$-adic expansions involving only digits $\ell/2+1, \ell/2+2, \dots, \ell-1$.
The remainder was ${0}$ union the essential class (using our notation).
Hence $\tau_\mu(q) = \min(\tau_A(q), \tau_B(q))$.
The probabilities were chosen so that $\tau_A(q) = \tau_B(q)$ had $\ell/2$ solution, and hence $\ell/2$ non-differential points in the $L^q$-spectrum.
The construction was done in such a way that very precise things could be said about $\tau_A(q)$, which was not the case in our construction.

Our construction is somewhat different, in that we subdivided $[0,1]$ into an arbitrary number of regions.
Our two points of non-differentiability came from solutions to $\tau_A(q) = \tau_B(q)$ and 
    for various choices of $A$ and $B$.
\end{remark}

\section{Observations and final comments}
\label{sec:final}

We will abuse notation and say that a set $A < B$ if for all $x \in A$ and for all $y \in B$ we have $x < y$.
Consider two adjacent $K_i, K_j$ with $K_i < K_j$.
One technique that we exploited was to ensure that $K_i$ and $K_j$ were sufficiently far apart
    so that where exists $R$ consecutive maps $S \in \mathcal{A}$ where $K_i < S((0,1)) < K_j$.
This helped to ensure that all points $x$ with $K_i < x < K_j$ were essential points.

Similarly if $\# \mathcal{B}_i > 1$ and $S_1, S_2 \in \mathcal{B}_i$ adjacent, 
    we again required $S_1$ and $S_2$ to be sufficiently separated 
    so that where exists $R$ consecutive maps $S \in \mathcal{A}$ where $S_1(K_i) < S((0,1)) < S_2(K_i)$.
This again helped to ensure that all points $x$ with $K_i < x < K_j$ were essential points.

In this paper our constructions used partitions $\mathcal{A} \cup \mathcal{B}_1 \cup \dots \cup \mathcal{B}_d$ 
    where all $\# \mathcal{B}_i = 1$ or all $\#\mathcal{B}_i = 2$.
It is relatively straightforward using these ideas to include other combinations and other
    constructions, and with possible large sets $\mathcal{B}_i$..
Our sets $\mathcal{K}_i$ were either singletons if $\# \mathcal{B}_i = 1$ 
    or a Cantor set with positive Hausdorff dimension if $\#\mathcal{B}_i = 2$.
It is possible to construct $\mathcal{B}_i$s such that the associate 
    $K_i$ is a countable set.
Other more complicated structures are also possible.
We would need to use the observations above to ensure that 
    $S_1(K_i)$ and $S_2(K_i)$ are sufficient separated for $S_1, S_2 \in 
    \mathcal{B}_i$.

Consider $\tau_\mu(q)$ in Example \ref{ex:Lq}.
In this example, we had two points of non-differentiability.
It {\em might} be possible to use this sort of construction to construct a $\mu$ where
    $\tau_\mu(q)$ has an arbitrary number of points of non-differentiability.
This would use a different construction than the one used in \cite{Testud}.
To do this, consider a curve $\tau_{K_i}(q)$.
We see that the slope of this curve is bounded by the local dimensions on $K_i$.
By choosing the probabilities associate to $K_i$ carefully, this is controlable.
Further, we see that the $y$-intercept of this curve is minus the Hausdorff dimension of 
    $K_i$.  
That is $\tau_{K_i}(0) = -\dim_H(K_i)$.
By choosing $\mathcal{B}_i$ so that $\#\mathcal{B}_i = t$, and still 
    satisfying the open set condition, we can get $\dim_H(K_i) = \log(t)/\log(R)$, and hence
    $\tau_{K_i}(0) = -\log(t)/\log(R)$.
There are most likely restricting on which values of $t$ that can be 
    chosen, and the appropriateness of the maps in $\mathcal{B}_i$ to 
    ensure requirements \eqref{req:1}, \eqref{req:4}, and \eqref{req:5} still hold.
Lastly, we can subtly alter the definition of $\mathcal{A}$ to be
    $\mathcal{A} = \{S_i\}_{i \in I}$ where $S_i(x) = x / R + \frac{i}{R^2 N}$
    and $I = \{i : \forall j\ S_i((0,1)) \cap K_j = \emptyset\}$.
The resulting measure should still have the necessary properties to allow this 
    type of decomposition.
By choosing $N$ large, the local dimensions in the essential class should be smoothed out.
That is, if $[a_N, b_N]$ is the set of attainable local dimensions of the essential class,
    we should have $a_N \nearrow 1$ and $b_n \searrow 1$ as $N \to \infty$.
If we can control the slope and $y$-intercept of $\tau_{K_i}(q)$, 
    and $\tau_E(q)$, it should be 
    possible to choose a sequence of $K_i$ where the $\tau_\mu(q) = \tau_{K_i}(q)$ on 
    different regions, leading to multiple points of non-differentiability.

Although not proven, in all of these constructions, we have that the set of attainable local
    dimensions of the essential class is a non-degenerate interval.
It is not clear if this is in fact always the case, or if it is possible 
    for the set of local dimensions of the essential class to be an isolated
    point, while at the same time having loop classes with other 
    local dimensions.

\section{Acknowledgements}

I would like to thank Alex Rutar for bringing \cite{Testud} to my attention.

\end{document}